\documentclass[12pt]{article}

\usepackage{fullpage}
\usepackage[utf8]{inputenc} %
\usepackage[T1]{fontenc}    %
\usepackage{url}            %
\usepackage{booktabs}       %
\usepackage{amsfonts}       %
\usepackage{nicefrac}       %
\usepackage{microtype}      %
\usepackage{mathrsfs} 

\usepackage{amsmath,amssymb,amsthm,xcolor}
\usepackage{array}
\usepackage{float}
\newtheorem{theorem}{Theorem}

\newtheorem{remark}{Remark}
\newtheorem{lemma}{Lemma}

\newtheorem{definition}{Definition}

\newtheorem{assumption}{Assumption}
\usepackage{graphicx}
\usepackage{subcaption}
\usepackage{mdframed}
\usepackage{lipsum}
\usepackage{wrapfig}
\usepackage{tikz}
\usepackage{enumitem}
\allowdisplaybreaks
\usepackage{epstopdf}
\usepackage{algorithmic}
\usepackage{algorithm}
\usepackage{booktabs}
\usepackage{changepage}
\usepackage{enumitem}
\usepackage{indentfirst}
\usepackage{stmaryrd} %
\usepackage{soul}
\usepackage{pgfplots}
\usetikzlibrary{arrows.meta}
\usepgfplotslibrary{fillbetween}
\usepackage[colorlinks=true,citecolor=blue,linkcolor=blue,urlcolor=blue]{hyperref}
\usepackage{cleveref}

\newcommand{\ip}[2]{\left\langle #1,#2\right\rangle}
\newcommand{\Var}{\operatorname{Var}}

\usepackage{amsmath,amsfonts,bm}

\DeclareMathOperator{\col}{col}
\DeclareMathOperator{\row}{row}
\newcommand*{\rom}[1]{%
\textup{\uppercase\expandafter{\romannumeral#1}}%
}

\def\eqref#1{equation~\ref{#1}}
\def\Eqref#1{Equation~\ref{#1}}

\def\1{\bm{1}}

\DeclareMathAlphabet{\mathsfit}{\encodingdefault}{\sfdefault}{m}{sl}
\SetMathAlphabet{\mathsfit}{bold}{\encodingdefault}{\sfdefault}{bx}{n}

\DeclareMathOperator*{\argmin}{arg\,min}

\DeclareMathOperator{\sign}{sign}

\newcommand{\cH}{\mathcal{H}}

\newcommand{\cK}{\mathcal{K}}

\newcommand{\cP}{\mathcal{P}}

\newcommand{\cW}{\mathcal{W}}

\newcommand{\cZ}{\mathcal{Z}}

\newcommand{\bE}{\mathbb{E}}

\newcommand{\bP}{\mathbb{P}}

\newcommand{\norm}[1]{\left\lVert#1\right\rVert}

\newcommand{\rank}{\mathrm{rank}}

\title{Certifying optimality in nonconvex robust PCA}
\author{\large Pinxi Gong\thanks{\texttt{gpx2005@connect.hku.hk}, Department of Mathematics, The University of Hong Kong.} \and  Lexiao Lai\thanks{\texttt{lai.lexiao@hku.hk}, Department of Mathematics, The University of Hong Kong.} \and Jianhao Ma\thanks{\texttt{jianhaom@wharton.upenn.edu}, Department of Statistics and Data Science, University of Pennsylvania.}}
\begin{document}
\maketitle

\begin{abstract}
    Robust principal component analysis seeks to recover a low-rank matrix from fully observed data with sparse corruptions. A scalable approach fits a low-rank factorization by minimizing the sum of entrywise absolute residuals, leading to a nonsmooth and nonconvex objective. Under standard incoherence conditions and a random model for the corruption support, we study factorizations of the ground-truth rank-$r$ matrix with both factors of rank $r$. With high probability, every such factorization is a Clarke critical point. We also characterize the local geometry: when the factorization rank equals $r$, these solutions are sharp local minima; when it exceeds $r$, they are strict saddle points.
\end{abstract}

\section{Introduction}
\label{sec:intro}

Robust principal component analysis (PCA) aims to recover a low-rank matrix
$L \in \mathbb{R}^{m \times n}$ from observations
\begin{equation*}
    M = L + S,
\end{equation*}
where $S$ is a sparse matrix modeling outliers.
Unlike classical PCA, which can be severely distorted by a small fraction of corrupted
entries, robust PCA explicitly accounts for such deviations and enables reliable extraction
of low-dimensional structure.
This problem has found widespread applications across many domains, including
computer vision (e.g., background subtraction \cite{lrslibrary2015,bouwmans2019}),
signal processing (e.g., image denoising \cite{xu2017denoising}),
and bioinformatics (e.g., gene expression analysis \cite{liu2013robust}).
Robust PCA can also be viewed as a fully observed special case of robust matrix completion,
where both missing data and sparse corruptions may be present \cite{candes2009,chen2021bridging}.

A classical approach to robust PCA is convex optimization.
The seminal work \cite{candes2011}
introduces \emph{Principal Component Pursuit} (PCP),
\begin{equation*}
\min_{L,S} \ \|L\|_{\star} + \lambda \|S\|_1
\quad \text{s.t.} \quad L + S = M,
\end{equation*}
where $\|A\|_\star$ is the nuclear norm and $\|A\|_1:= \sum_{i = 1}^m \sum_{j = 1}^n |A_{ij}|$ is the vectorized $\ell_1$-norm.
They show that PCP enjoys exact recovery guarantees under incoherence and sparsity conditions.
Algorithmically, PCP can be solved in polynomial time using augmented Lagrangian or ADMM-based methods \cite{LinLiuSu2011}, which iteratively call singular value decompositions (SVDs) as a subroutine. However, repeatedly computing SVDs can become computationally expensive at scale.

Motivated by scalability, there has been substantial interest in \emph{nonconvex} formulations \cite{netrapalli2014altproj,yi2016fast,cai2019accelerated,chen2021bridging},
which represent the low-rank matrix $L$ through low-dimensional factors $XY^{\top}$ to avoid repeated spectral decompositions.
An $\ell_1$-based factorized formulation for robust PCA is
\begin{equation}
\label{eq:obj}
\min_{X \in \mathbb{R}^{m \times k},\, Y \in \mathbb{R}^{n \times k}}
f(X,Y) := \left\| X Y^\top - M \right\|_{1}.
\end{equation}
Although \Eqref{eq:obj} is nonconvex and nonsmooth, it is widely used in practice and often performs
well empirically under sparse corruptions \cite{KeKanade2005,ErikssonVanDenHengel2010,MengXuZhangZhao2013}.

\subsection{Empirical motivation: convergence and scalability}
\label{subsec:empirical}

Before developing a theory, we illustrate two empirical phenomena that motivate our analysis.
First, Figure~\ref{fig:errors} shows that a simple subgradient method applied to minimize the factorized objective (\Eqref{eq:obj}) with small initialization
can drive the reconstruction error rapidly to zero, both in the exactly-parameterized setting ($k=r:=\mathrm{rank}(L)$) and in the
overparameterized setting ($k>r$).
\begin{figure}[ht]
    \centering
    \begin{subfigure}[t]{0.48\textwidth}
        \centering
        \includegraphics[width=\textwidth]{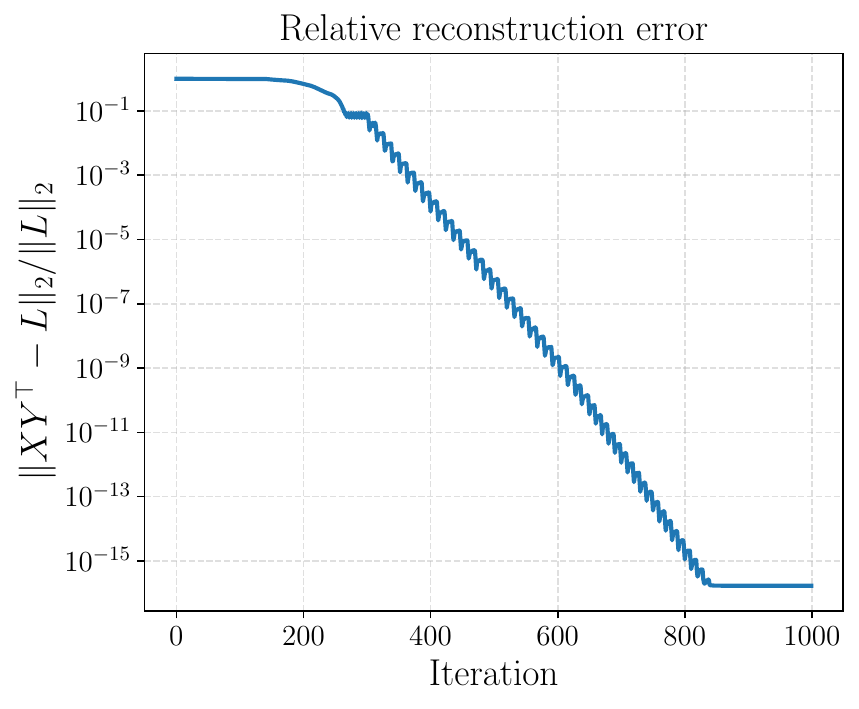}
        \caption{$m = n = 100$, $k = r = 10$.}
        \label{fig:error-exact}
    \end{subfigure}
    \hfill
    \begin{subfigure}[t]{0.48\textwidth}
        \centering
        \includegraphics[width=\textwidth]{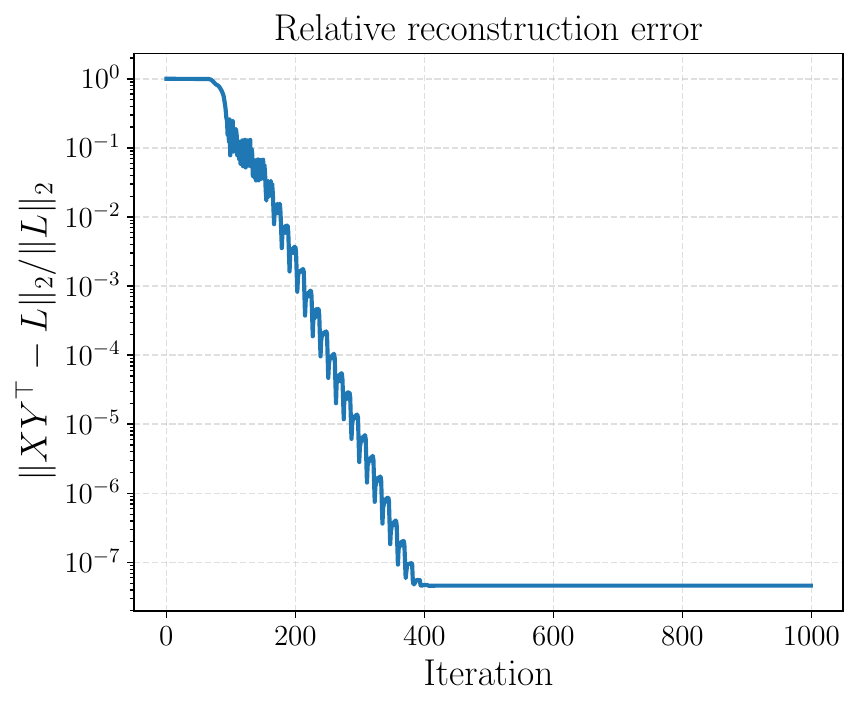}
        \caption{$m = n = 100$, $k = 10$, $r = 5$.}
        \label{fig:error-over}
    \end{subfigure}
    \caption{Convergence of the subgradient method with small initialization for solving \Eqref{eq:obj}.
    Each entry of $S$ is nonzero with probability $p = 0.1$. In all experiments, we adopt a robust exponentially decaying step-size schedule: the step size is reduced by a factor of $0.5$ whenever the objective value fails to decrease for $10$ consecutive iterations.}
    \label{fig:errors}
\end{figure}

Second, in larger-scale synthetic experiments, minimizing \Eqref{eq:obj} is often substantially faster than both convex approaches, such as PCP, and competitive with nonconvex alternatives like \cite{yi2016fast,cai2019accelerated,chen2021bridging}, while also achieving superior accuracy.
Table~\ref{tab:synthetic} presents representative comparisons in terms of runtime and recovery error, and Figure~\ref{fig:rpca_side_by_side} illustrates success rates over a range of ranks and corruption levels.
These encouraging results motivate a deeper investigation into the structural properties of \Eqref{eq:obj}.
\begin{table}[ht]
  \centering
  \caption{Robust PCA experiments with $m = 10000$, $n = 8000$, $k = 400$, and $r = 200$, averaged over 10 random instances.
  Accuracy is reported as $\|L_{\text{est}}-L\|_2/\|L\|_2$.
 Implementations of methods 2) and 3) are from \cite{lrslibrary2015} and \cite{cai_accaltproj_rpca}, respectively. We implemented methods 4) and 5) by ourselves according to the pseudo-codes presented in the respective papers \cite{yi2016fast,chen2021bridging}, as there is no publicly available implementation.}
  \label{tab:synthetic}
  \begin{tabular}{lcc}
    \toprule
    \textbf{Method} & \textbf{Time (s.)} & \textbf{Accuracy} \\
    \midrule
    \small{1) Minimizing \Eqref{eq:obj} by the subgradient method} & \small{\hphantom{0}199.4} & \small{\hphantom{11}8.21e-10} \\
    \small{2) Principal component pursuit \cite{candes2011} solved by IALM \cite{LinLiuSu2011}} & \small{\hphantom{0}479.1} & \small{\hphantom{0}1.48e-8} \\
    \small{3) Accelerated alternating projection \cite{cai2019accelerated}} & \small{\hphantom{0}196.0} & \small{\hphantom{0}2.68e-4} \\
    \small{4) Fast RPCA via gradient descent \cite{yi2016fast}} & \small{\hphantom{0}254.1} & \small{\hphantom{0}7.75e-2} \\
    \small{5) Nonconvex alternating minimization \cite{chen2021bridging}} & \small{\hphantom{0}321.0} & \small{\hphantom{0}4.64e-1} \\
    \bottomrule
  \end{tabular}
\end{table}

\begin{figure}[ht]
  \centering
  \begin{subfigure}[t]{0.49\textwidth}
    \centering
    \includegraphics[width=\linewidth]{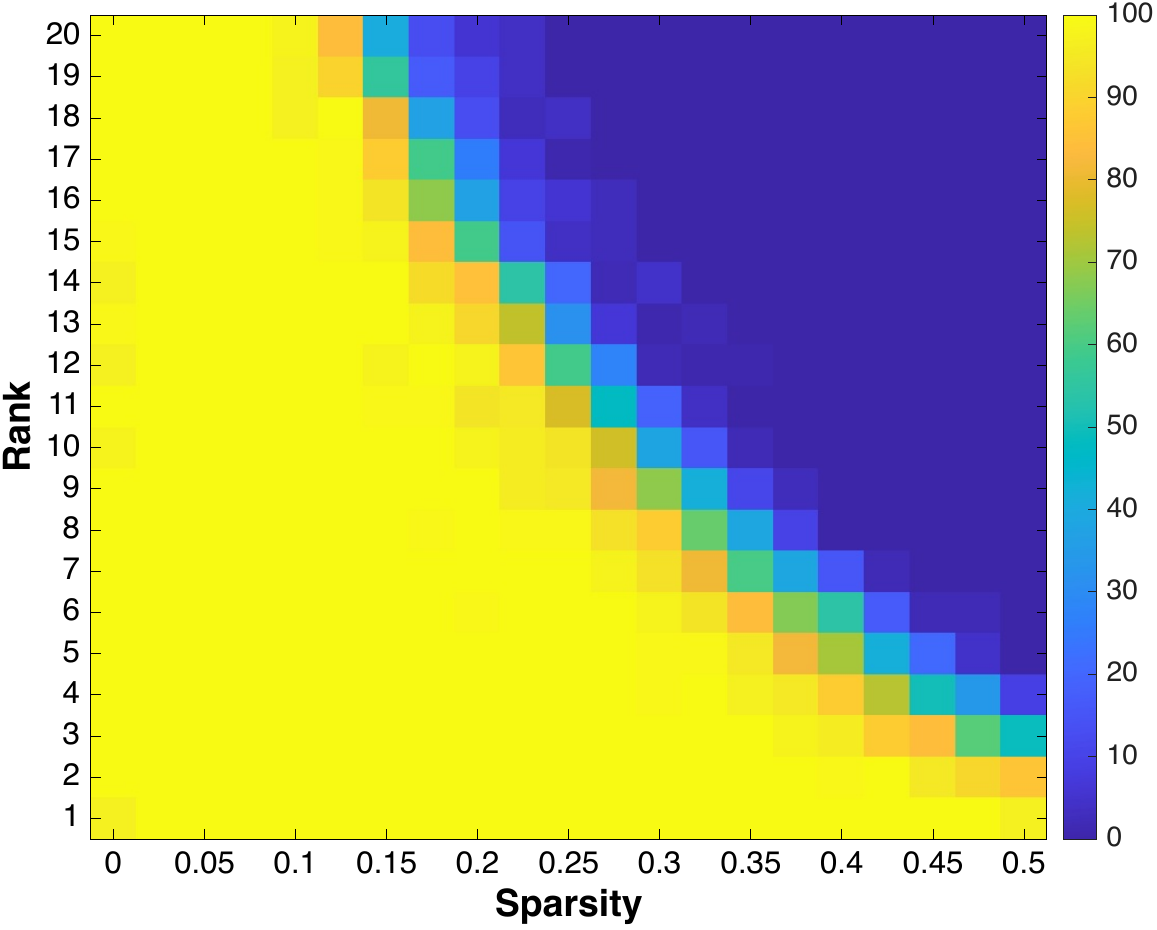}
    \caption{Factorized robust PCA (\Eqref{eq:obj}) with search rank $k = 20$ solved by the subgradient method.}
    \label{fig:rpca_subgrad}
  \end{subfigure}
  \hfill
  \begin{subfigure}[t]{0.49\textwidth}
    \centering
    \includegraphics[width=\linewidth]{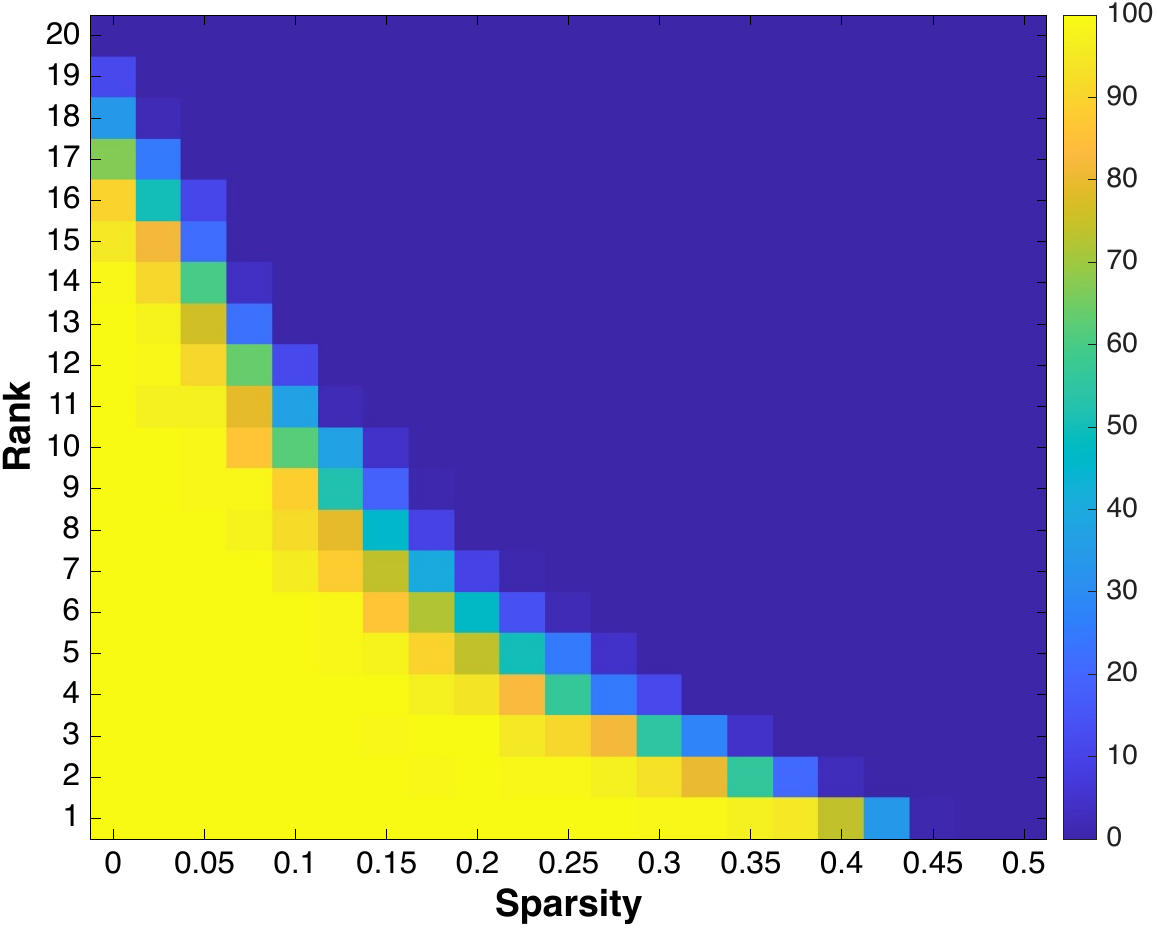}
    \caption{PCP solved by IALM.}
    \label{fig:rpca_ialm}
  \end{subfigure}
  \caption{Recovery success-rate heatmaps (in \%) for robust PCA on synthetic low-rank plus sparse matrices.
 The matrices are of size $100$ by $80$.
  Each grid point averages over $100$ independent trials;
  The success criterion is $\|L_{\mathrm{est}}-L\|_2/\|L\|_2 \le 10^{-3}$.}
  \label{fig:rpca_side_by_side}
\end{figure}

\subsection{Main results}
\label{subsec:main-results}

Despite the strong empirical performance demonstrated in Section \ref{subsec:empirical}, the theoretical understanding of \Eqref{eq:obj} remains limited. The objective is nonsmooth and nonconvex, and finding a global solution is NP-hard in general \cite{gillis2018complexity}. Most existing landscape results apply only to the rank-one, noiseless regime (i.e., $k = r = 1$ and $S = 0$). In that setting, prior work shows that all local minima of \Eqref{eq:obj} are global \cite{josz2018theory,josz2022nonsmooth,guan2024ell_1,josz2025nonsmooth}, and that the ground-truth solutions form sharp local minima \cite{charisopoulos2021low}. Our aim, by contrast, is to provide optimality guarantees for the \emph{true} factorizations in the general-rank, noisy setting.

We work under the following standard assumption in the robust PCA literature \cite{candes2011}.

\begin{assumption}
\label{assumption:standing}
The observed matrix $M \in \mathbb{R}^{m \times n}$ is the sum of a low-rank matrix $L$ and a sparse matrix $S$, i.e., $M = L + S$, where:
\begin{enumerate}
    \item \textbf{Low-rank structure and incoherence:} The matrix $L$ has rank $r$ and is $\mu$-incoherent: let $L = U \Sigma V^\top$ be its singular value decomposition, then
    \begin{equation}
    \label{eq:incoh}
    \max_i \|U_i\|_2 \leq \sqrt{\frac{\mu r}{m}}
    \quad\text{and}\quad
    \max_j \|V_j\|_2 \leq \sqrt{\frac{\mu r}{n}},
    \end{equation}
    where $U_i$ and $V_j$ denote the $i$-th and $j$-th rows of $U$ and $V$, respectively.

    \item \textbf{Sparse corruption model:} The entries of $S$ are independent, and each entry is nonzero with probability $p > 0$.
\end{enumerate}
\end{assumption}

We study the optimality of the set of true solutions with rank-$r$ factors:
\begin{equation*}
\mathcal L:= \left\{(X,Y)\in \mathbb{R}^{m\times k}\times \mathbb{R}^{n \times k}:\ 
XY^{\top} = L\ \text{and}\ \mathrm{rank}(X) = \mathrm{rank}(Y) = r\right\}.
\end{equation*}
When $k=r$, the rank condition is redundant and $\mathcal L=\{(X,Y):XY^\top=L\}$.
However, when $k > r$, the factorization space includes redundant parameterizations, and some pairs $(X, Y)$ satisfying $XY^{\top} = L$ may have higher rank and need not be critical points of the objective (see, e.g., \cite{ma2025can}). For this reason, we restrict our attention to $\mathcal{L}$.

We can now state our main theorem. It certifies that points in $\mathcal L$ are always first-order critical, and
moreover distinguishes the second-order geometry between the exactly-parameterized and overparameterized regimes.

\begin{theorem}
\label{thm:main}
Let Assumption~\ref{assumption:standing} hold with
\begin{equation*}
p\leq \frac{c}{\log(mn)}
\quad\text{and}\quad
r\leq \min\left\{c\sqrt{\frac{\min\{m, n\}}{\mu \log^2(mn)}},\,k\right\},
\end{equation*}
where $c>0$ is a sufficiently small universal constant.
Then there exists a universal constant $C>0$ such that, with probability at least $1-\exp\{-Cr\log^2(mn)\}$, any $(X^{\star},Y^{\star})\in \mathcal L$ is a
(Clarke) critical point of $f$.
Moreover, it is a local minimum when $k = r$, and a strict saddle point\footnote{We say that a point is a strict saddle point if it is a critical point that violates the second-order necessary condition for optimality \cite[Theorem~13.24]{rockafellar2009variational}.}
when $k>r$.
\end{theorem}
\Cref{thm:main} provides a geometric explanation for why first-order methods can succeed near the ground truth.
In the exact-rank case ($k=r$), the true solution set forms local minima, which is consistent with the rapid
error decay observed in Figure~\ref{fig:error-exact} under small initialization.
In contrast, when $k>r$, the same factorizations become strict saddle points: intuitively, the additional degrees
of freedom create descent directions that do not change the product $XY^\top$ to first-order, but can decrease
the nonsmooth objective by aligning with corruption patterns. Understanding why subgradient methods can still converge to such strict saddle points remains an open question for future work.

Moreover, the conditions on the corruption probability $p$ and the true rank $r$ in \Cref{thm:main} are quite general. The corruption probability $p$ may be a constant up to logarithmic factors, which is nearly tight compared with the guarantees established for convex formulations \cite{candes2011}. In contrast, our rank condition allows the true rank $r$ to be on the order of $\sqrt{\min\{m,n\}}$, rather than $\min\{m,n\}$ as in the convex setting. While this requirement may be conservative, it still covers a broad class of practically relevant problems. We view sharpening this rank condition to match the convex benchmark as an interesting direction for future work.

We present the proof of \Cref{thm:main} in Section~\ref{sec:proof}. At a high level, our main technical contribution lies in reducing the first-order optimality analysis to bounding an associated \emph{empirical process}. This reformulation enables us to leverage powerful tools from probability and statistical learning theory \cite{pollard1990empirical,sen2018gentle,talagrand2014upper,ledoux2013probability,wellner2013weak}. It also allows us to establish the result under more general conditions, as our assumptions on the corruption probability and the rank are milder than those required by direct analyses of the problem.

At the same time, \Cref{thm:main} does \emph{not} assert that the ground truth is globally optimal for all
sparse perturbations. Indeed, even though points in $\mathcal L$ enjoy the stated local optimality/saddle properties,
global minimizers of \Eqref{eq:obj} can correspond to solutions that overfit corruptions.

\begin{remark}[On global optimality of true solutions]
\label{rem:global-opt}
\Cref{thm:main} cannot, in general, be strengthened to global optimality over all sparse perturbations $S$.
Indeed, for any $L \neq 0$ and any nonempty support $\Omega \subset \{1,\dots,m\}\times\{1,\dots,n\}$,
there exists $S$ supported on $\Omega$ such that $\inf_{X,Y} f(X,Y)<\|S\|_1$.
\end{remark}

\section{Proof of \Cref{thm:main}}\label{sec:proof}
We begin by recalling some standard notation. For $p\in [1,\infty]$, $\|\cdot\|_p$ denotes the usual $\ell_p$-norm on vectors; when applied to a matrix, it refers to the $\ell_p$-norm of the vector obtained by stacking all entries. Given $A\in \mathbb R^{m\times n}$ and $p,q\in [1,\infty]$, we write $\|A\|_{p,q}$ for its $(p,q)$-matrix norm defined by
\[\|A\|_{p,q}:= \sup_{\|x\|_p\le 1}\|Ax\|_q.\]
We define the operator norm of a matrix $A\in \mathbb{R}^{m\times n}$ as $\norm{A}_{\mathrm{op}}:=\norm{A}_{2, 2}$. For notational convenience, we denote by $d_p$ the $\ell_p$-norm, $d_{p,q}$ the $(p,q)$-matrix norm, and $d_{\mathrm{op}}$ the operator norm, respectively.
 For $A,B\in \mathbb R^{m\times n}$, we let $\langle A,B\rangle := \mathrm{tr}(A^\top B)$ be their Frobenius inner product. For $\cW\subset \mathbb R^{m\times n}$, we denote by
\[
d(A,\cW):= \inf_{B\in \cW} \|A - B\|_2,
\qquad
\mathcal P_{\cW}(A):= \argmin_{B\in \cW} \|A - B\|_2
\]
the Euclidean (Frobenius) distance from $A$ to $\cW$ and the corresponding metric projection. We also denote the column space, row space, and kernel of $A$ by $\col(A)$, $\row(A)$, and $\ker(A)$, respectively. Given $f(n)$ and $g(n)$, we write $f(n)=O(g(n))$ when there exists a universal constant $C>0$ satisfying $f(n) \leq Cg(n)$ for all large enough $n$. For convenience, we also use the shorthand $f(n)\lesssim g(n)$ to denote
$f(n)=O(g(n))$.

In this section, we study the factorized robust PCA objective~(\Eqref{eq:obj}) with $M:= L+S$. For the low-rank component $L$ with $\operatorname{rank}(L)=r$, let $L = U \Sigma V^\top$ be its singular value decomposition with $U\in \mathbb R^{m\times r}$, $V\in \mathbb R^{n\times r}$, and $\Sigma\in \mathbb R^{r\times r}$. Consider the linear subspace
\begin{equation*}%
	T :=\left\{ H V^\top + U K^\top :H \in \mathbb{R}^{m \times r}, K \in \mathbb{R}^{n \times r} \right\}.
\end{equation*}
The same linear subspace was considered in previous works of low-rank matrix recovery (see, e.g., \cite{candes2009,candes2011}). For any true factorization $(X^{\star},Y^{\star})\in \mathcal L$, as $\col(X^{\star}) = \col(U)$ and $\col(Y^{\star}) = \col(V)$, we have
\begin{equation*}%
   	T = \left\{ H (Y^{\star})^\top + (X^{\star}) K^\top :H \in \mathbb{R}^{m \times k}, K \in \mathbb{R}^{n \times k} \right\}. 
\end{equation*}
For the sparse component $S$, we denote by $\Omega \subset \{1,\dots,m\}\times\{1,\dots,n\}$ its support. We write $\mathcal P_\Omega$ for the orthogonal projector onto the linear subspace of matrices supported on $\Omega$. We denote by $\Omega^c$ the complement of $\Omega$, and by $\mathcal P_{\Omega^c}$ the projector onto matrices supported on $\Omega^c$.

This section is organized as follows. In Section \ref{sec:first_order}, we certify first-order optimality by reducing the analysis to a restricted $\ell_1$ operator-norm bound for $\mathcal P_\Omega$ on the linear subspace $T$. The key probabilistic estimate is \Cref{thm::first-order-optimality}, whose proof occupies the remainder of that subsection. In Section \ref{sec:second_order}, we study second-order behavior: when $k=r$ we show that points in $\mathcal L$ are sharp local minima (Theorem~\ref{thm:localmin}), whereas in the overparameterized regime $k>r$ they fail to be local minima and instead are strict saddle points (Theorem~\ref{thm:over}). All these results together lead to \Cref{thm:main}.
\subsection{Certifying first-order optimality}\label{sec:first_order}
In this subsection, we characterize first-order optimality of $f$ in the sense of Clarke. The Clarke subdifferential of $f$ \cite{clarke1975,clarke1990} is defined for all $(X,Y) \in \mathbb{R}^{m\times k} \times \mathbb{R}^{n\times k}$ by
\begin{equation*}
\begin{aligned}
\partial f(X,Y) := \{ (P,Q) \in \mathbb{R}^{m\times k} \times \mathbb{R}^{n\times k} ~:~ &
f^\circ((X,Y);(H,K)) \ge \langle (P,Q),(H,K) \rangle, \\
& \forall (H,K)\in \mathbb{R}^{m\times k} \times \mathbb{R}^{n\times k} \}.
\end{aligned}
\end{equation*}
where $f^\circ$ denotes the generalized directional derivative defined by
\begin{equation*}
    f^\circ((X,Y);(H,K)) := \limsup_{\scriptsize\begin{array}{c} (\bar{X},\bar{Y})\rightarrow (X,Y) \\
    t \searrow 0
    \end{array}
    } \frac{f(\bar{X}+tH,\bar{Y}+tK)-f(\bar{X},\bar{Y})}{t}.
\end{equation*}
By virtue of \cite[2.3.10 Chain Rule II]{clarke1990}, it holds that
\begin{equation*}
\partial f(X,Y) = \left\{ 
        (\Lambda Y, \Lambda^{\top} X):\Lambda \in \text{sign}(X Y^{\top} - M)
    \right\},
\end{equation*}
where $\sign$ denotes the elementwise sign mapping defined by $\sign(t) := t/|t|$ if $t\neq 0$, and $\sign(0) := \left[-1,1\right]$. By the generalized Fermat's rule \cite[Proposition 2.3.2]{clarke1990}, a point $(X,Y)\in \mathbb{R}^{m\times k} \times \mathbb{R}^{n\times k}$ is a local minimum of $f$ only if $0\in \partial f(X,Y)$. We say that $(X,Y)$ is a (Clarke) critical point if $0\in \partial f(X,Y)$.

The following lemma reduces the task of checking the criticality of solutions in $\mathcal L$ to upper-bounding the $\ell_1$ operator norm of $\mathcal P_\Omega$ when it is restricted to the linear subspace $T$.

\begin{lemma}\label{lemma:interior}
If there exists $\epsilon>0$ such that $$\|\mathcal{P}_\Omega(W)\|_1 \leq \frac{1-\epsilon}{2} \|W\|_1,\quad\forall W\in T,$$ then there exists $\Lambda\in \sign(-S)$ with $\|\mathcal{P}_{\Omega^c}(\Lambda)\|_{\infty} \leq 1-\epsilon$ such that $\Lambda V = 0$ and $\Lambda^{\top} U = 0$. Consequently, any $(X^{\star},Y^{\star})\in \mathcal L$ is a critical point of $f$. 
\end{lemma}
\begin{proof}
	It suffices to prove the following inclusion
	\begin{equation}\label{eq:inclusion}
		0\in \{(\Lambda V,\Lambda^{\top} U)\in \mathbb{R}^{m\times r}\times \mathbb{R}^{n\times r}: \Lambda \in \sign(-S), \|\mathcal{P}_{\Omega^c}(\Lambda)\|_{\infty} \leq 1-\epsilon\}.
	\end{equation}
	We denote the set on the right-hand side as $\mathcal{D} \subset \mathbb{R}^{m\times r}\times \mathbb{R}^{n\times r}$ and let $\sigma_\mathcal{D}: \mathbb{R}^{m\times r}\times \mathbb{R}^{n\times r}\to \mathbb{R}$ be its support function, defined by $\sigma_\mathcal{D}(H,K) :=\sup_{(A,B)\in \mathcal{D}} \langle A,H\rangle+ \langle B,K\rangle$. Let $\Xi\subset \mathbb R^{m\times n}$ be the feasible region of $\Lambda$ given by \Eqref{eq:inclusion}. Since $\mathcal{D}$ is convex, it contains $0$ if and only if $\sigma_\mathcal{D}(H,K)\ge 0$ for all $(H,K)\in \mathbb{R}^{m\times r}\times \mathbb{R}^{n\times r}$. Fix any $(H,K)$. It holds that
\begin{align*}
   \sigma_\mathcal{D}(H,K) =~&\sup_{(A,B)\in \mathcal{D}} \langle A,H\rangle+ \langle B,K\rangle\\
   =~&\sup_{\Lambda \in \Xi}~ \langle \Lambda V, H \rangle + \langle \Lambda^{\top} U, K \rangle  \\
    =~& \sup_{\Lambda \in \Xi}~ \langle \Lambda , HV^{\top} + UK^{\top} \rangle \\
    =~&\sup_{\Lambda \in \Xi}~ \langle \mathcal P_{\Omega^c}(\Lambda), \mathcal P_{\Omega^c}(HV^{\top} + UK^{\top}) \rangle+\langle \mathcal P_{\Omega}(\Lambda), \mathcal P_{\Omega}(HV^{\top} + UK^{\top}) \rangle\\
    \stackrel{(a)}{\geq} ~& (1-\epsilon)\|\mathcal{P}_{\Omega^c}(HV^{\top} + U K^{\top})\|_1 - \|\mathcal{P}_{\Omega}(HV^{\top} + U K^{\top})\|_1\\
    =~&(1-\epsilon)\|HV^\top + U K^{\top}\|_1 - (2-\epsilon)\|\mathcal{P}_{\Omega}(HV^\top + U K^{\top})\|_1\\
    \stackrel{(b)}{\ge}~& (1-\epsilon - (2-\epsilon)(1-\epsilon)/2)\|HV^\top + U K^{\top}\|_1\\
    \geq~& 0.
\end{align*}
Here in $(a)$, we use Hölder's inequality and the condition that $\|\mathcal{P}_{\Omega^c}(\Lambda)\|_{\infty} \leq 1-\epsilon$. In $(b)$, we use the condition $\|\mathcal{P}_\Omega(W)\|_1 \leq \frac{1-\epsilon}{2} \|W\|_1,\forall W\in T$. 
Lastly, any $(X^{\star},Y^{\star})\in \mathcal L$ is a critical point as $\col(X^{\star}) = \col(U)$ and $\col(Y^{\star}) = \col(V)$.
\end{proof}
The following theorem provides the desired uniform upper bound on the restricted $\ell_1$ operator norm of $\mathcal P_\Omega$ under Assumption \ref{assumption:standing}. Its proof requires several auxiliary lemmas, which we recall in Appendix \ref{sec:auxiliary}.
\begin{theorem}
\label{thm::first-order-optimality}
    Let Assumption \ref{assumption:standing} hold. Then with probability at least $1-\exp\{-Cr\log^2(mn)\}$, for all $W\in T$, we have
    \begin{equation}\label{eq:restricted_bound}
        \norm{\cP_\Omega(W)}_1\leq C\left(\sqrt{p\log(mn)}+\frac{\log^2(mn)\mu r^2}{\min\{m, n\}}+\log(mn)r\sqrt{\frac{\mu}{\min\{m, n\}}}\right)\norm{W}_1,
    \end{equation}
    where $C>0$ is a universal constant. Consequently, when $p\lesssim 1/\log(mn)$ and $r\lesssim \sqrt{\frac{\min\{m, n\}}{\mu\log^2(mn)}}$, there exists a constant $\epsilon>0$ such that, with the same probability, we have
    \begin{equation*}
        \norm{\cP_\Omega(W)}_1\leq \frac{1-\epsilon}{2}\norm{W}_1, \quad \forall W\in T.
    \end{equation*}
\end{theorem}
\subsubsection{Proof of \Cref{thm::first-order-optimality}}
To prove the desired upper bound \Eqref{eq:restricted_bound}, we define
\begin{equation*}
    \cW := \bigl\{ W\in T : \|W\|_1=1 \bigr\},
\end{equation*}
and for each $W\in \cW$,
\begin{equation*}
    X_W:=\|\cP_{\Omega}(W)\|_1 - p\|W\|_1=\sum_{i,j}(\Omega_{ij}-p)|W_{ij}|,
\end{equation*}
where each $\Omega_{ij}$ is an independent Bernoulli random variable with parameter $p$; that is, $\Omega_{ij} = 1$ with probability $p$, and $\Omega_{ij} = 0$ otherwise. Note that $\bE [X_W]=0$ for all $W\in \mathbb{R}^{m\times n}$. Hence, to derive \Eqref{eq:restricted_bound}, it suffices to study the following random process
\begin{equation*}
    \sup_{W\in\cW} X_W.
\end{equation*}
\begin{sloppypar} \noindent To this goal, we show that this random process satisfies the mixed subgaussian–subexponential increments condition (\Cref{def::subgaussian-subexponential}), thereby enabling the application of generic chaining bounds (\Cref{thm:mixed-tail}).
\end{sloppypar}

\noindent\textit{Step 1: check the mixed-tail increment condition.}
For any $W, W'\in \cW$, we have
\begin{equation*}
    X_W-X_{W'}=\sum_{i,j}\underbrace{(\Omega_{ij}-p)\left(|W_{ij}|-|W'_{ij}|\right)}_{:=Z_{ij}},
\end{equation*}
where $\bE [Z_{ij}]=0$, $|Z_{ij}|\le \|W-W'\|_\infty$, and
\begin{equation*}
    \sum_{i,j}\Var(Z_{ij})=p(1-p)\sum_{i,j}\left(|W_{ij}|-|W'_{ij}|\right)^2\le p\sum_{i,j}\left(W_{ij}-W'_{ij}\right)^2= p\|W-W'\|_2^2.
\end{equation*}
Therefore, applying Bernstein's inequality (Lemma~\ref{lem::bernstein}) yields, for all $u\ge0$,
\begin{equation*}%
\mathbb{P}\Bigl( |X_W-X_{W'}|  \ge  \sqrt{u}\cdot  \sqrt{2p}\norm{W-W'}_2 + u\cdot  \frac{1}{3}\norm{W-W'}_\infty \Bigr)  \le  2e^{-u}.
\end{equation*}
Thus, the random process $(X_W)_{W\in\cW}$ has mixed subgaussian/subexponential increments with respect to $(\sqrt{2p}\norm{\cdot}_2,\frac{1}{3}\norm{\cdot}_\infty)$.

\noindent\textit{Step 2: upper bounds on diameters.} We next establish upper bounds for the diameters of $\cW$ under the Frobenius norm $\|\cdot\|_2$ and the $\ell_\infty$-norm $\|\cdot\|_\infty$, denoted $\Delta_2(\cW)$ and $\Delta_\infty(\cW)$, respectively.
\begin{lemma}
\label{lemma:diameter_bound}
The diameters of $\cW$ under the Frobenius and $\ell_\infty$-norms satisfy:
\begin{equation*}
    \begin{aligned}
        \Delta_2(\cW) &\le 2 \sqrt{\mu r \left( \frac{1}{m} + \frac{1}{n} \right)},\\
        \Delta_\infty(\cW) &\le 2 \mu r \left( \frac{1}{m} + \frac{1}{n} \right).
    \end{aligned}
\end{equation*}
\end{lemma}

\begin{proof}
Let \(\cP_U := UU^\top\) and \(\cP_V := VV^\top\) denote the orthogonal projectors onto \(\col(U)\) and \(\col(V)\), respectively. For any \(i\in \{1,\ldots,m\}\) and \(j\in \{1,\ldots,n\}\), we define
\begin{equation*}
u_i := \cP_U e_i \in \mathbb{R}^m, \qquad v_j := \cP_V e_j \in \mathbb{R}^n.
\end{equation*}
Then, by \(\mu\)-incoherence (\Eqref{eq:incoh}) and Cauchy-Schwarz inequality, we have
\begin{equation}
\label{eq:incoh-l2-linf}
\|u_i\|_2^2 = e_i^\top \cP_U e_i = \|U^\top e_i\|_2^2 \le \frac{\mu r}{m}.
\end{equation}
Similarly, we have $\|v_j\|_2^2 \le \frac{\mu r}{n}$.
Furthermore, we also have
\begin{equation*}
    |( \cP_U )_{ab}| = |\langle U^\top e_a, U^\top e_b\rangle| \le \frac{\mu r}{m},\quad
|( \cP_V )_{cd}| \le \frac{\mu r}{n}.
\end{equation*}

\noindent\textbf{Frobenius diameter.}
Fix \(W\in\cW\). Using the canonical basis \(E_{ij}:=e_i e_j^\top\) and linearity of \(\cP_T\), we have
\begin{equation*}
W = \cP_T(W) = \cP_T\Big(\sum_{i,j} W_{ij} E_{ij}\Big)
\;=\; \sum_{i,j} W_{ij} \cP_T(E_{ij}).
\end{equation*}
By the triangle inequality, we obtain
\begin{equation*}
    \begin{aligned}
        \|W\|_2 &\le \sum_{i,j} |W_{ij}| \|\cP_T(E_{ij})\|_2\le \Big(\sum_{i,j}|W_{ij}|\Big) \max_{i,j}\|\cP_T(E_{ij})\|_2= \max_{i,j}\|\cP_T(E_{ij})\|_2,
    \end{aligned}
\end{equation*}
where we use \(\|W\|_1=1\) in the last step. For any $i\in \{1,\ldots,m\}, j\in \{1,\ldots,n\}$, we have
\begin{equation}
\label{eq:PT-on-basis}
\cP_T(E_{ij}) = u_i e_j^\top + e_i v_j^\top - u_i v_j^\top.
\end{equation}
Denote \(a := \|u_i\|_2^2\) and \(b := \|v_j\|_2^2\) for short. A direct expansion of the inner product yields
\begin{equation*}
\begin{aligned}
\|\cP_T(E_{ij})\|_2^2
&= \|u_i e_j^\top\|_2^2 + \|e_i v_j^\top\|_2^2 + \|u_i v_j^\top\|_2^2
\\
&\quad + 2\langle u_i e_j^\top, e_i v_j^\top\rangle
- 2\langle u_i e_j^\top, u_i v_j^\top\rangle
- 2\langle e_i v_j^\top, u_i v_j^\top\rangle \\
&\stackrel{(a)}{=} a + b + ab + 2(u_i^\top e_i)(e_j^\top v_j) - 2a\cdot(e_j^\top v_j) - 2b\cdot(e_i^\top u_i)\\
&=a + b - ab\\
&\le a + b\\
&\leq \frac{\mu r}{m} + \frac{\mu r}{n}.
\end{aligned}
\end{equation*}
Here, in $(a)$, we use the facts that \(u_i^\top e_i = \|u_i\|_2^2 = a\) and \(e_j^\top v_j = \|v_j\|_2^2 = b\). Therefore, we derive that
\begin{equation*}
\sup_{W\in\cW}\|W\|_2 \le \max_{i,j}\|\cP_T(E_{ij})\|_2\le \sqrt{\frac{\mu r}{m} + \frac{\mu r}{n}},
\end{equation*}
which implies
\begin{equation*}
\Delta_2(\cW) = 2 \sup_{W\in\cW}\|W\|_2
\le 2 \sqrt{\mu r\left(\frac{1}{m}+\frac{1}{n}\right)}.
\end{equation*}

\noindent\textbf{Entrywise \(\ell_\infty\) diameter.}
Fix \(W\in\cW\) and indices \((k,\ell)\in \{1,\ldots,m\}\times \{1,\ldots,n\}\). Since \(W\in T\) and \(\cP_T\) is the orthogonal projector onto \(T\), we have
\begin{equation*}
W_{k\ell} = \langle E_{k\ell}, W\rangle =\langle E_{k\ell}, \cP_T(W)\rangle = \langle \cP_T(E_{k\ell}), W\rangle.
\end{equation*}
By duality of \(\|\cdot\|_1\) and \(\|\cdot\|_\infty\), we have
\begin{equation*}
|W_{k\ell}| \le \|\cP_T(E_{k\ell})\|_\infty  \|W\|_1 = \|\cP_T(E_{k\ell})\|_\infty.
\end{equation*}
Thus, \(\|W\|_\infty \le \max_{k,\ell} \|\cP_T(E_{k\ell})\|_\infty\). Now, it suffices to bound $\|\cP_T(E_{k\ell})\|_\infty$ uniformly over all $k,\ell$.
From \Eqref{eq:PT-on-basis}, the \((i,j)\)-th entry of \(\cP_T(E_{k\ell})\) is
\begin{equation}
\label{eq:entrywise-form}
\big(\cP_T(E_{k\ell})\big)_{ij} = (u_k)_i\delta_{\ell j} + \delta_{k i}(v_\ell)_j - (u_k)_i(v_\ell)_j,
\end{equation}
where \(\delta\) denotes the Kronecker delta, defined by $\delta_{ij}=1$ if $i=j$, and $\delta_{ij}=0$ otherwise. Set \(a:=\|u_k\|_2^2 \le \frac{\mu r}{m}\) and \(b:=\|v_\ell\|_2^2 \le \frac{\mu r}{n}\). Using \Eqref{eq:incoh-l2-linf}, each coordinate satisfies
\(
|(u_k)_i| = |(\cP_U)_{i k}| \le \frac{\mu r}{m}
\)
and
\(
|(v_\ell)_j| \le \frac{\mu r}{n}.
\)
We now bound \(|(\cP_T(E_{k\ell}))_{ij}|\) by considering the following four cases:
\begin{itemize}
    \item \textbf{Case 1: $i=k, j=\ell$.} We have
    \begin{equation*}
        |(u_k)_k + (v_\ell)_\ell - (u_k)_k (v_\ell)_\ell|
= |a + b - ab| \le a + b.
    \end{equation*}
    \item \textbf{Case 2: $i=k, j\neq \ell$.} We have
    \begin{equation*}
        |(v_\ell)_j - (u_k)_k (v_\ell)_j|
= |(v_\ell)_j||1 - (u_k)_k| \le |(v_\ell)_j|(1 + |(u_k)_k|) \le b + ab \le a + b.
    \end{equation*}
    \item \textbf{Case 3: $i\neq k, j=\ell$.} We have
    \begin{equation*}
        |(u_k)_i - (u_k)_i (v_\ell)_\ell|
\le |(u_k)_i|(1 + |(v_\ell)_\ell|) \le a + ab \le a + b.
    \end{equation*}
    \item \textbf{Case 4: $i\neq k, j\neq \ell$.} We have
    \begin{equation*}
        |(u_k)_i (v_\ell)_j| \le ab \le a + b.
    \end{equation*}
\end{itemize}
Here, in the above inequalities, we use \(0\le a,b\le 1\) repeatedly, which holds since they are squared norms of projections. Therefore, for all \(i,j\), we have
\begin{equation*}
\big|\big(\cP_T(E_{k\ell})\big)_{ij}\big| \le a + b \le \frac{\mu r}{m} + \frac{\mu r}{n},
\end{equation*}
which implies
\begin{equation*}
\|\cP_T(E_{k\ell})\|_\infty \le \frac{\mu r}{m} + \frac{\mu r}{n}\qquad\text{for all } k,\ell.
\end{equation*}
Taking the maximum over \((k,\ell)\) and recalling \(\|W\|_1=1\), we obtain
\begin{equation*}
\sup_{W\in\cW}\|W\|_\infty \le \frac{\mu r}{m} + \frac{\mu r}{n}\quad
\implies\quad
\Delta_\infty(\cW) = 2 \sup_{W\in\cW}\|W\|_\infty \le 2\mu r\left(\frac{1}{m}+\frac{1}{n}\right).
\end{equation*}
This completes the proof.
\end{proof}

\noindent\textit{Step 3: bounding the $\gamma_2$-functional via Gaussian process.}
By the majorizing measure theorem (\Cref{lem::majorizing-measure}), we have $\gamma_2(\cW,d_2)\lesssim \bE \left[\sup_{W\in\cW} Y_W\right]$ where $(Y_W)_{W\in \cW}$ is a centered Gaussian process defined by
$Y_W= \langle G,W\rangle$ with $G$ an $m\times n$ matrix of i.i.d. $\mathcal N(0,1)$ entries. For the Gaussian process $(Y_W)_{W\in \cW}$, note that
\begin{equation*}
    \sup_{W\in\cW} Y_W
= \sup_{\substack{W\in T\\ \|W\|_1=1}}\langle G,W\rangle
\le  \sup_{\|W\|_1=1}\langle G,W\rangle
= \|G\|_\infty.
\end{equation*}
Applying the maximal inequality (\Cref{lem::maximal-ineq}), we have 
\begin{equation*}
    \bE\left[\sup_{W\in\cW} Y_W\right]\leq \bE\left[ \|G\|_\infty\right]\lesssim \sqrt{\log(mn)}.
\end{equation*}
Therefore, we have
\begin{equation*}%
\gamma_2(\cW,d_2)\lesssim\ \bE\left[\sup_{W\in\cW} Y_W\right]\lesssim \sqrt{\log(mn)}.
\end{equation*}

\noindent\textit{Step 4: bounding the $\gamma_1$-functional via Dudley's integral.}
First, we invoke Lemma~\ref{lem::dudley} to bound the $\gamma_1$-functional via Dudley’s entropy integral:
\begin{equation*}
    \gamma_1(\cW, d_\infty)\leq C \int_0^{\infty}\log\left(N(\cW, d_\infty, \varepsilon)\right)\operatorname{d}\!\varepsilon,
\end{equation*}
where $C>0$ is a universal constant.
To proceed, we decompose the set $\cW$ by introducing the following three sets:
\begin{equation*}
    \mathcal{H}:=\left\{H\in \mathbb{R}^{m\times r}: \norm{H}_{2, 1}\leq \sqrt{\frac{\mu r}{n}}\right\}, \quad \mathcal{K}:=\left\{K\in \mathbb{R}^{n\times r}: \norm{K}_{2, 1}\leq \sqrt{\frac{\mu r}{m}}\right\},
\end{equation*}
and 
\begin{equation*}
    \mathcal{Z}:=\left\{Z\in \mathbb{R}^{r\times r}: \norm{Z}_{2}\leq \frac{\mu r}{\sqrt{mn}}\right\}.
\end{equation*}
We now bound the metric entropy $\log( N(\cW, d_\infty, \varepsilon))$ by relating it to the metric entropies of $\mathcal{H}$, $\mathcal{K}$, and $\mathcal{Z}$ under appropriate norms.
\begin{lemma}
    The metric entropy $\log( N(\cW, d_\infty, \varepsilon))$ satisfies the following bound:
    \begin{equation*}
        \begin{aligned}
            \log(N(\cW, d_\infty, \varepsilon))&\leq \log\left(N\left(\mathcal{H}, d_{2, \infty}, \frac{1}{3}\sqrt{\frac{n}{\mu r}}\varepsilon\right)\right) + \log\left(N\left(\mathcal{K}, d_{2, \infty}, \frac{1}{3}\sqrt{\frac{m}{\mu r}}\varepsilon\right)\right)\\
            &\quad+ \log\left(N\left(\mathcal{Z}, d_{\mathrm{op}}, \frac{1}{3}\frac{\sqrt{mn}}{\mu r}\varepsilon\right)\right).
        \end{aligned}
\end{equation*}
\end{lemma}
\begin{proof}
    For any $W\in \cW$, we have the following decomposition
    \begin{equation*}
        W=\underbrace{WV}_{:=H}V^{\top}+U\underbrace{U^{\top}W}_{:=K^{\top}} -U\underbrace{U^{\top}WV}_{:=Z}V^{\top}.
    \end{equation*}
    We have $H\in \mathcal{H}$ since
    \begin{equation*}
        \norm{H}_{2, 1}=\sum_{i=1}^{m}\norm{W_{i,:}V}_2\leq \sum_{i=1}^{m}\sum_{j=1}^{n}|W_{i, j}|\norm{e_j^{\top}V}_2\leq \sqrt{\frac{\mu r}{n}}\norm{W}_1\leq \sqrt{\frac{\mu r}{n}}. 
    \end{equation*}
    Similarly, we also have $K\in \mathcal{K}$. For $Z$, we have $Z\in \mathcal{Z}$ since
    \begin{equation*}
        \norm{Z}_2=\norm{\sum_{i=1}^{m}\sum_{j=1}^{n}W_{i,j}U_{i, :}^{\top}V_{j,:}}_2\leq \sum_{i=1}^{m}\sum_{j=1}^{n}|W_{i,j}|\norm{U_{i, :}}_2\norm{V_{j,:}}_2\leq \frac{\mu r}{\sqrt{mn}}\norm{W}_1\leq \frac{\mu r}{\sqrt{mn}}.
    \end{equation*}
    To proceed, let $\cH_{\varepsilon_1}$ be an $\varepsilon_1$-net of $\cH$ under the $d_{2,\infty}$-metric, with $\varepsilon_1=\frac{1}{3}\sqrt{\frac{n}{\mu r}}\varepsilon$. Similarly, let $\cK_{\varepsilon_2}$ be an $\varepsilon_2$-net of $\cK$ under $d_{2,\infty}$ with $\varepsilon_2=\frac{1}{3}\sqrt{\frac{m}{\mu r}}\varepsilon$, and let $\cZ_{\varepsilon_3}$ be an $\varepsilon_3$-net of $\cZ$ under the operator norm with $\varepsilon_3=\frac{\sqrt{mn}}{3\mu r}\varepsilon$.  For any $W\in\cW$ with decomposition $W=HV^\top+UK^\top-UZV^\top$, let $H'$, $K'$, and $Z'$ denote the closest elements to $H$, $K$, and $Z$ in the covering nets $\cH_{\varepsilon_1}$, $\cK_{\varepsilon_2}$, and $\cZ_{\varepsilon_3}$, respectively. We also denote $W'=H'+K'^{\top}-Z'$. Then, we have
    \begin{equation*}
        \norm{W-W'}_{\infty}\leq \norm{(H-H')V^{\top}}_{\infty}+\norm{U(K-K')^{\top}}_{\infty}+\norm{U(Z-Z')V^{\top}}_{\infty}.
    \end{equation*}
    For the first term, we have
    \begin{equation*}
        \norm{(H-H')V^{\top}}_{\infty}=\max_{i, j}|\ip{(H-H')_{i,:}}{V_{j, :}}|\leq \sqrt{\frac{\mu r}{n}}\norm{H-H'}_{2, \infty}\leq \frac{1}{3}\varepsilon.
    \end{equation*}
    Similarly, we have $\norm{U(K-K')^{\top}}_{\infty}\leq \sqrt{\frac{\mu r}{m}}\norm{K-K'}_{2, \infty}\leq \frac{1}{3}\varepsilon$. Lastly, we have
    \begin{equation*}
        \norm{U(Z-Z')V^{\top}}_{\infty}=\max_{i, j}|\ip{U_{i,:}(Z-Z')}{V_{j, :}}|\leq \frac{\mu r}{\sqrt{mn}}\norm{Z-Z'}_{\mathrm{op}}\leq \frac{1}{3}\varepsilon.
    \end{equation*}
    Combining the three bounds gives $\|W-W'\|_\infty\le \varepsilon$. Thus, the set $\cW_{\varepsilon}=\{W'=H'+K'^{\top}-Z'\}$ is an $\varepsilon$-net of $\cW$ with respect to the $d_\infty$-metric. This completes the proof.
\end{proof}

Equipped with this lemma, we have
\begin{equation*}
    \begin{aligned}
        \int_0^{\infty}\log\left(N(\cW, d_\infty, \varepsilon)\right)\operatorname{d}\!\varepsilon&\leq  3\sqrt{\frac{\mu r}{n}}\underbrace{\int_0^{\infty}\log\left(N(\mathcal{H}, d_{2, \infty}, \varepsilon)\right)\operatorname{d}\!\varepsilon}_{(\rom{1})}\\
        &\quad + 3\sqrt{\frac{\mu r}{m}}\underbrace{\int_0^{\infty}\log\left(N(\mathcal{K}, d_{2, \infty}, \varepsilon)\right)\operatorname{d}\!\varepsilon}_{(\rom{2})}\\
        &\quad + 3\frac{\mu r}{\sqrt{mn}}\underbrace{\int_0^{\infty}\log\left(N(\mathcal{Z}, d_{\mathrm{op}}, \varepsilon)\right)\operatorname{d}\!\varepsilon}_{(\rom{3})}.
    \end{aligned}
\end{equation*}
We first control the first term $(\rom{1})$. To this goal, we estimate the metric entropy $\log\left(N(\mathcal{H}, d_{2, \infty}, \varepsilon)\right)$ in two regimes.
\begin{lemma}
    For any $0 < \varepsilon \le \sqrt{\mu r / n}$, the metric entropy of $\mathcal H$ under the $d_{2,\infty}$ metric admits the following bounds
    \begin{equation*}
        \log\left(N(\mathcal{H}, d_{2, \infty}, \varepsilon)\right)\lesssim \begin{cases}
mr\log(1/\varepsilon) & \text{(volume bound)}; \\
\sqrt{\frac{\mu r^3}{n\varepsilon^2}}\log(1/\varepsilon)  & \text{(sparsity bound)}.
\end{cases} 
    \end{equation*}
    
\end{lemma}
\begin{proof}
We prove the above two bounds separately.

\medskip
\noindent\textbf{Volume bound.}
Note that $\|H\|_{2,\infty} = \max_{1\leq i\leq m} \|H_{i,:}\| \le \sqrt{\mu r/n}$.
    Hence, it suffices to construct $m$ $\varepsilon$-nets for each row in $d_{2}$-metric and the entropy can be upper-bounded by
    \begin{equation*}
        \begin{aligned}
            \log\left(N(\mathcal{H}, d_{2, \infty}, \varepsilon)\right)&\leq m \log\left(N\left(\mathbb{B}_r\left(\sqrt{\mu r/n}\right), d_{2}, \varepsilon\right)\right)\\
            &\leq m\log\left(\left(3\sqrt{\frac{\mu r}{n\varepsilon^2}}\right)^r\right)\\
            &\leq mr \log(1/\varepsilon).
        \end{aligned}
    \end{equation*}
    Here, in the last inequality, we use the fact that $\mu \leq n/r$ by the definition of incoherence.
    
\medskip
\noindent\textbf{Sparsity bound.}
In this case, we exploit the row sparsity. To this end, we define
\begin{equation*}
    S(H) := \bigl\{i : \|H_{i,:}\|_2 > \varepsilon\bigr\}.
\end{equation*}
Since $\norm{H}_{2, 1}\leq \sqrt{\mu r/n}$, we have
\begin{equation*}
    s := |S(H)| \le  \frac{\sqrt{\mu r/n}}{\varepsilon}
    = \sqrt{\frac{\mu r}{n\varepsilon^2}}.
\end{equation*}
Then, we cover $\mathcal{H}$ in three steps: for any $H\in \mathcal{H}$, we cover the large rows on $S(H)$, and note that the remaining rows are already small.

\medskip
\noindent\emph{Substep 4.1: Choose the large-row support.}
The number of subsets $S\subset\{1,\dots,m\}$ with $|S|=s$ satisfies
\begin{equation*}
    \binom{m}{s} \le \left(\frac{em}{s}\right)^s.
\end{equation*}

\medskip
\noindent\emph{Substep 4.2: Cover the large rows.}
For each $i\in S(H)$ we have $\|H_{i,:}\|_2\le \sqrt{\mu r/n}$, and we only need accuracy $\varepsilon$ in $\|\cdot\|_2$ for each such row. As above,
\begin{equation*}
    N\bigl(B_r(\sqrt{\mu r/n}), d_2, \varepsilon\bigr) 
    \le \left(3\sqrt{\frac{\mu r}{n\varepsilon^2}}\right)^r.
\end{equation*}

\medskip
\noindent\emph{Substep 4.3: Small rows.}
For $i\notin S(H)$ we have $\|H_{i,:}\|_2\le \varepsilon$. Thus, each such row lies within distance $\varepsilon$ (in $\|\cdot\|_2$) of the zero vector, so we can approximate all these rows by $0$ without increasing $d_{2,\infty}$ by more than $\varepsilon$. Hence, they do not contribute to the covering number.

\medskip
Putting Substeps 4.1 and 4.2 together, we obtain
\begin{equation*}
    N(\mathcal{H}, d_{2,\infty}, \varepsilon)
    \le \binom{m}{s}\left(3\sqrt{\frac{\mu r}{n\varepsilon^2}}\right)^{rs},
\end{equation*}
so
\begin{equation*}
    \begin{aligned}
        \log N(\mathcal{H}, d_{2,\infty}, \varepsilon)
    \le s\log\left(\frac{em}{s}\right) + rs\log\left(\frac{1}{\varepsilon}\right)\leq \sqrt{\frac{\mu r^3}{n\varepsilon^2}}\log\left(\frac{1}{\varepsilon}\right).
    \end{aligned}
\end{equation*}
This completes the proof.
\end{proof}

Now we are ready to control the first term $(\rom{1})$. Upon setting $\varepsilon_0 = \sqrt{\frac{\mu r}{m^2 n}}$, the integral can be split at $\varepsilon_0$ as follows
\begin{equation*}
    \begin{aligned}
        (\rom{1})&=\int_0^{\sqrt{\frac{\mu r}{n}}}\log\left(N(\mathcal{H}, d_{2, \infty}, \varepsilon)\right)\operatorname{d}\!\varepsilon\\
        &\lesssim \int_0^{\varepsilon_0} mr\log(1/\varepsilon)\operatorname{d}\!\varepsilon+\int_{\varepsilon_0}^{\sqrt{\frac{\mu r}{n}}} \sqrt{\frac{\mu r^3}{n\varepsilon^2}}\log(1/\varepsilon)\operatorname{d}\!\varepsilon\\
        &=mr\left(\varepsilon_0\log(1/\varepsilon_0)+\varepsilon_0\right)+\frac{1}{2}\sqrt{\frac{\mu r^3}{n}}\left(\log^2(1/\varepsilon_0)-\frac{1}{4}\log^2\left(\frac{n}{\mu r}\right)\right)\\
        &\lesssim\log^2(mn)\sqrt{\frac{\mu r^3}{n}}.
    \end{aligned}
\end{equation*}
Similarly, we can upper-bound the second term $(\rom{2})$
\begin{equation*}
    (\rom{2})\leq C\log^2(mn)\sqrt{\frac{\mu r^3}{m}}.
\end{equation*}
Lastly, we control the third term $(\rom{3})$. 
First, note that for any matrix $A \in \mathbb{R}^{r \times r}$, we have $\norm{A}_{\mathrm{op}} \leq \norm{A}_2$. Consequently, any $\varepsilon$-cover of the set $\mathcal{Z}$ with respect to the Frobenius metric $d_2$ is also an $\varepsilon$-cover with respect to the operator metric $d_{\mathrm{op}}$. This implies that
\begin{align*}
N(\mathcal{Z}, d_{\mathrm{op}}, \varepsilon) \leq N(\mathcal{Z}, d_2, \varepsilon).
\end{align*}
Using the standard volume bound for the covering number of a Euclidean ball, we have:
\begin{align*}
N(\mathcal{Z}, d_2, \varepsilon) \leq \left(\frac{3R}{\varepsilon}\right)^{r^2},
\end{align*}
where $R=\frac{\mu r}{\sqrt{mn}}$.
Thus, we bound the third term $(\rom{3})$ as follows:
\begin{align*}
(\rom{3}) \leq \int_0^{R} \log \left[ \left(\frac{3R}{\varepsilon}\right)^{r^2} \right] \operatorname{d}\!\varepsilon = \frac{\mu r^3}{\sqrt{mn}} \int_0^{1} \ln \left(\frac{3}{\varepsilon}\right) \operatorname{d}\!\varepsilon\leq C\frac{\mu r^3}{\sqrt{mn}}.
\end{align*}
Therefore, putting everything together yields
\begin{equation*}
    \begin{aligned}
        \int_0^{\infty}\log\left(N(\cW, d_\infty, \varepsilon)\right)\operatorname{d}\!\varepsilon&\leq  \sqrt{\frac{\mu r}{n}}\cdot(\rom{1})+\sqrt{\frac{\mu r}{m}}\cdot(\rom{2})+\frac{\mu r}{\sqrt{mn}}\cdot (\rom{3})\\
        &\leq C\log^2(mn)\mu r^2\left(\frac{1}{m}+\frac{1}{n}\right)+C\frac{\mu^2r^4}{mn}\\
        &\leq C\log^2(mn)\mu r^2\left(\frac{1}{m}+\frac{1}{n}\right),
    \end{aligned}
\end{equation*}
provided that $r=O\Big(\sqrt{\frac{\min\{m, n\}}{\mu\log^2(mn)}}\Big)$.
This implies
\begin{equation*}
    \gamma_1(\cW,d_{\infty}) =C\log^2(mn)\mu r^2\frac{1}{\min\{m, n\}}.
\end{equation*}

\noindent\textit{Step 5: Putting it together.}
Finally, applying \Cref{thm:mixed-tail}, for any fixed $W_0\in \cW$, we have
\begin{equation*}%
    \begin{aligned}
        \bE  \left[\sup_{W\in\cW} X_W\right]
 &=\bE  \left[\sup_{W\in\cW} X_W-X_{W_0}\right]\\
 &\leq \bE  \left[\sup_{W\in\cW} \left|X_W-X_{W_0}\right|\right]\\
 &\le  C\bigl(\sqrt{p}\gamma_2(\cW,d_2)+\gamma_1(\cW,d_{1})\bigr)+2  \sup_{W \in \cW} \bE \left[|X_W - X_{W_0}|\right]\\
 &\leq C\left(\sqrt{p\log(mn)}+\mu r^2\left(\frac{1}{m}+\frac{1}{n}\right)\log^2(mn)\right),
    \end{aligned}
\end{equation*}
for an absolute constant $C>0$. Here, in the last inequality, we use the following estimation 
\begin{equation*}
        \bE \left[|X_W - X_{W_0}|\right]\leq 2\bE \left[|X_W|\right]\leq 2\left(\Var(X_W)\right)^{1/2}\leq 2\sqrt{p}\norm{W}_2\leq 2\sqrt{p}.
\end{equation*}
On the other hand, note that $\Delta_2(\cW) \le 2 \sqrt{\mu r \left( \frac{1}{m} + \frac{1}{n} \right)}$ and $\Delta_\infty(\cW) \le 2 \mu r \left( \frac{1}{m} + \frac{1}{n} \right)$. Hence, applying \Cref{thm:mixed-tail} yields that 
\begin{equation*}
    \begin{aligned}
        \bP\bigg(\sup_{W\in\cW} X_W &\geq C\left(\sqrt{p\log(mn)}+\log^2(mn)\mu r^2\frac{1}{\min\{m, n\}}\right)\\
        &\qquad+c\left(\sqrt{ \frac{\mu ru}{\min\{m,n\}} } +\frac{\mu ru}{\min\{m,n\}}\right)\!\bigg) \leq e^{-u}.
    \end{aligned}
\end{equation*}
Upon setting $u=Cr\log^2(mn)$, with probability at least $1-\exp\{-Cr\log^2(mn)\}$, we have
    \begin{equation*}
        \sup_{W\in\cW} X_W\leq C\left(\sqrt{p\log(mn)}+\frac{\log^2(mn)\mu r^2}{\min\{m, n\}}+\log(mn)r\sqrt{\frac{\mu}{\min\{m, n\}}}\right).
    \end{equation*}
This completes the proof.
\subsection{Determination of second-order optimality}\label{sec:second_order}
In this section, we investigate second-order optimality of solutions in $\mathcal L$. When the rank of the underlying low-rank component $L$ equals the factorization rank $k$, any $(X^{\star},Y^{\star})\in \mathcal L$ is in fact a sharp local minimum: the objective value grows at least linearly with the distance to $\mathcal L$. This is the object of Theorem \ref{thm:localmin}, which leverages the restricted $\ell_1$ operator-norm estimate for $\mathcal P_{\Omega}$ established in \Cref{thm::first-order-optimality}. Given $A\in \mathbb R^{m\times n}$ and $r\le \rank(A)$, we denote by $\sigma_r(A)$ the $r$-th largest singular value of $A$.
\begin{theorem}\label{thm:localmin}
	Assume $k = r$ and that there exists $\epsilon>0$ such that 
    \begin{equation}\label{eq:restricted_bound_sharp}
	\|\mathcal{P}_\Omega(W)\|_1 \leq \frac{1-\epsilon}{2} \|W\|_1,\quad\forall W\in T.        
    \end{equation}
	Then for any $(X^{\star},Y^{\star})\in \mathcal L$, there exists $\rho > 0$ such that
	\begin{align*}
		f(X,Y) - f(X^{\star},Y^{\star}) \ge \frac{\epsilon}{8}\min\{\sigma_{r}(X^{\star}),\sigma_{r}(Y^{\star})\}\, d((X,Y),\mathcal L)
	\end{align*}
	for all $(X,Y)\in B((X^{\star},Y^{\star}),\rho)$.
\end{theorem}

\begin{remark}[Implied linear convergence of the subgradient method]
As the objective (\Eqref{eq:obj}) is weakly convex \cite{vial1983strong}, the sharp growth of $f$ guaranteed by Theorem \ref{thm:localmin} implies that the subgradient method with geometrically decaying step sizes converges linearly when initialized near the solution set $\mathcal L$; see, for example, \cite{burke1993weak,davis2018subgradient,li2020nonconvex}. This corroborates our numerical experiments reported in Figure \ref{fig:error-exact}.
\end{remark}
\begin{proof}
 Fix $(X^{\star},Y^{\star})\in \mathcal L$. As $k = r$, $\mathcal L$ is a smooth embedded manifold. By the tubular neighborhood theorem \cite{lee2003smooth}, there is a neighborhood $\mathcal{N}$ of $(X^{\star},Y^{\star})$ so that the projection
$\mathcal P_{\mathcal L}$ onto $\mathcal L$ is single-valued and continuous on
$\mathcal{N}$. Fix $(X,Y)\in \mathcal N$ and set
\[
(X^\sharp,Y^\sharp):=\mathcal P_{\mathcal L}(X,Y),\qquad H:=X-X^\sharp,\quad K:=Y-Y^\sharp.
\]
Let $G_\sharp:=\{(X^\sharp A, -Y^\sharp A^\top): A\in\mathbb R^{r\times r}\}$ be the tangent 
space to $\mathcal L$ at $(X^\sharp,Y^\sharp)$. Since $(X^\sharp,Y^\sharp)$ minimizes
$\frac12\|(X',Y')-(X,Y)\|_2^2$ over $(X',Y')\in\mathcal L$, first-order optimality yields $(H,K)\in G_\sharp^\perp$. Also,
\begin{equation*}%
d:= d\big((X,Y),\mathcal L\big) = \|(H,K)\|_2.
\end{equation*}
Expand
\[
\Delta:= XY^\top - X^{\star}Y^{\star\top}
= \underbrace{X^\sharp K^\top + H(Y^\sharp)^\top}_{\Delta_1} + \underbrace{HK^\top}_{\Delta_2}.
\]
Note that $\Delta_1\in T$ and $\Delta_2$ is quadratic in $(H,K)$. We have
\begin{align*}
	f(X,Y) - f(X^{\star},Y^{\star}) &= \|XY^{\top} - M\|_1 - \|X^{\star}(Y^{\star})^{\top} - M\|_1\\
	&= \|\Delta - S\|_1 - \|S\|_1\\
    &= \|\mathcal{P}_{\Omega^c}(\Delta)\|_1 + \|\mathcal{P}_{\Omega}(\Delta) - S\|_1 - \|S\|_1\\
	&= \|\mathcal{P}_{\Omega^c}(\Delta)\|_1 + \langle -\sign(S),\mathcal{P}_\Omega(\Delta)\rangle\\
	&\ge \epsilon \|\mathcal{P}_{\Omega^c}(\Delta)\|_1+ \langle  \Lambda,\mathcal{P}_{\Omega^c}(\Delta)\rangle  + \langle -\sign(S),\mathcal{P}_\Omega(\Delta)\rangle\\
	&=\epsilon \|\mathcal{P}_{\Omega^c}(\Delta)\|_1+ \langle  \Lambda,\Delta\rangle,
\end{align*}
where $\Lambda\in -\sign(S)$ satisfies $\|\mathcal{P}_{\Omega^c}(\Lambda)\|_\infty \leq 1- \epsilon$. By \Cref{lemma:interior}, we may assume additionally that $\Lambda Y^\sharp =\Lambda V =  0$ and $\Lambda^{\top} X^\sharp = \Lambda^{\top} U= 0$. Thus,
\begin{align*}
	\langle  \Lambda,\Delta\rangle = \langle  \Lambda,\Delta_1 + \Delta_2\rangle = \langle  \Lambda,\Delta_2\rangle.
\end{align*}
Also,
\begin{align*}
	\|\mathcal{P}_{\Omega^c}(\Delta)\|_1 \geq \|\mathcal{P}_{\Omega^c}(\Delta_1)\|_1 - \|\mathcal{P}_{\Omega^c}(\Delta_2)\|_1,
\end{align*}
where
\begin{align*}
	\|\mathcal{P}_{\Omega^c}(\Delta_1)\|_1 &= \|\Delta_1\|_1 - \|\mathcal{P}_{\Omega}(\Delta_1)\|_1\\
	&\ge \Bigl(1-\frac{1-\epsilon}{2}\Bigr)\|\Delta_1\|_1\\
	&= \frac{1+\epsilon}{2}\|\Delta_1\|_1
\end{align*}
by \Eqref{eq:restricted_bound_sharp}. Combining the above inequalities, we have
\begin{equation}\label{eq:overall}
    \begin{aligned}
f(X,Y)-f(X^\ast,Y^\ast)
&\ge \epsilon \|\mathcal P_{\Omega^c}(\Delta)\|_1 + \langle \Lambda,\Delta\rangle \\
&\ge \frac{\epsilon}{2}\|\Delta_1\|_1 - \epsilon\|\mathcal P_{\Omega^c}(\Delta_2)\|_1
  + \langle \Lambda,\Delta_2\rangle\\
  &\ge \frac{\epsilon}{2}\|\Delta_1\|_1 - \epsilon\|\mathcal P_{\Omega^c}(\Delta_2)\|_1
  -(1-\epsilon)\|\Delta_2\|_1\\
  &\ge \frac{\epsilon}{2}\|\Delta_1\|_1 - \|\Delta_2\|_1.        
    \end{aligned}
\end{equation}

We next obtain an explicit lower bound on $\|\Delta_1\|_1$ in terms of $d$ and $(X^{\star},Y^{\star})$. Since $(H,K)\in G_\sharp^\perp$, we have
\[
0 = \langle (H,K),(X^\sharp A,-Y^\sharp A^\top)\rangle
    = \langle H,X^\sharp A\rangle - \langle K,Y^\sharp A^\top\rangle = \langle H^\top X^\sharp - K^\top Y^\sharp,A\rangle
    \quad
\]
for all $A\in\mathbb R^{r\times r}$. Hence,
\begin{equation}\label{eq:orth}
H^\top X^\sharp = K^\top Y^\sharp.
\end{equation}
We first bound $\|\Delta_1\|_2$ from below. Note that
\begin{align*}
\|\Delta_1\|_2^2
&= \|X^\sharp K^\top + HY^{\sharp\top}\|_2^2 \\
&= \|X^\sharp K^\top\|_2^2 + \|HY^{\sharp\top}\|_2^2
   + 2\langle X^\sharp K^\top,HY^{\sharp\top}\rangle.
\end{align*}
By \Eqref{eq:orth}, the cross term satisfies
\begin{align*}
\langle X^\sharp K^\top,HY^{\sharp\top}\rangle
&= \operatorname{tr}\big((X^\sharp K^\top)^\top HY^{\sharp\top}\big)
 = \operatorname{tr}\big(KX^{\sharp\top} HY^{\sharp\top}\big) \\
&= \operatorname{tr}\big(Y^{\sharp\top} K X^{\sharp\top} H\big)
 = \operatorname{tr}\big((X^{\sharp\top} H)(X^{\sharp\top} H)^\top\big)
 = \|X^{\sharp\top} H\|_2^2 \;\ge\; 0.
\end{align*}
Thus,
\begin{equation*}
\|\Delta_1\|_2^2
 \ge \|X^\sharp K^\top\|_2^2 + \|HY^{\sharp\top}\|_2^2 \ge \sigma_{r}(X^\sharp)^2 \|K\|_2^2 + \sigma_{r}(Y^\sharp)^2 \|H\|_2^2 \ge \sigma_\sharp^2 d^2,
\end{equation*}
where
\(
\sigma_\sharp := \min\{\sigma_{r}(X^\sharp),\sigma_{r}(Y^\sharp)\}.
\)
It follows that
\begin{equation}\label{eq:Delta1-lower}
\|\Delta_1\|_1 \;\ge\; \|\Delta_1\|_2 \;\ge\; \sigma_\sharp\, d.
\end{equation}

Next we relate $\sigma_\sharp$ to $(X^{\star},Y^{\star})$. Define
\[
\sigma_{\star} := \min\{\sigma_{r}(X^{\star}),\sigma_{r}(Y^{\star})\}.
\]
As $(X^{\star},Y^{\star})\in\mathcal L$ and $k=r$, both $X^{\star}$ and $Y^{\star}$ have rank $r$, so $\sigma_{\star}>0$. The map
\[
\mathcal L\ni (X,Y)\ \longmapsto\ \min\{\sigma_{r}(X),\sigma_{r}(Y)\}
\]
is continuous, and the projection $\mathcal P_{\mathcal L}$ is continuous on $\mathcal N$ and coincides with the identity map on $\mathcal L$. Therefore, after possibly shrinking $\mathcal N$, we may assume
\[
\sigma_\sharp = \min\{\sigma_{r}(X^\sharp),\sigma_{r}(Y^\sharp)\}
 \ge \frac{\sigma_{\star}}{2}
\quad\text{for all }(X,Y)\in\mathcal N, (X^\sharp,Y^\sharp) = \mathcal P_{\mathcal L}(X,Y).
\]
Combining this with \Eqref{eq:Delta1-lower}, we obtain
\begin{equation}\label{eq:Delta1_d}
  \|\Delta_1\|_1 \;\ge\; \frac{\sigma_{\star}}{2}\, d
\quad\text{for all }(X,Y)\in\mathcal N.  
\end{equation}
By Cauchy--Schwarz inequality,
\begin{equation}\label{eq:Delta2_d}
 \|\Delta_2\|_1 = \|HK^\top\|_1 \le \sqrt{mn} \|H\|_2 \|K\|_2
 \le \frac{\sqrt{mn}}{2} \big(\|H\|_2^2+\|K\|_2^2\big)
 = \frac{\sqrt{mn}}{2} d^2.   
\end{equation}

Plugging \Eqref{eq:Delta1_d} and \Eqref{eq:Delta2_d} into \Eqref{eq:overall}, we have
\[
f(X,Y)-f(X^{\star},Y^{\star})
 \ge \frac{\epsilon \sigma_{\star}}{4}   d -  \frac{\sqrt{mn}}{2}  d^2 \ge  \frac{\epsilon \sigma_{\star}}{8}   d
\]
whenever $d\le \epsilon\sigma_{\star}/(4\sqrt{mn})$. This proves the theorem by taking $\rho\le \epsilon\sigma_{\star}/(4\sqrt{mn})$ such that $B((X^{\star},Y^{\star}),\rho)\subset \mathcal N$. 
\end{proof}

We next treat the case of overparametrization, namely, when the rank of $L$ is strictly less than the search rank $k$. Somewhat surprisingly, the solutions in $\mathcal L$ cease to be local minima even in the mildly overparameterized regime $k = r + 1$. This complements \cite[Theorem 6]{ma2025can}, which proves that they are not local minima when $k>2r$. Furthermore, by exhibiting directions of descent for any solution in $\mathcal L$, we show that these points are not second-order optimal in the sense of \cite[Theorem~13.24]{rockafellar2009variational}.

\begin{theorem}\label{thm:over}
Let $k>r\ge 1$, $(X^{\star},Y^{\star})\in \mathcal L$, and assume $S\neq 0$. Then there exist $(H,K)\in\mathbb{R}^{m\times k}\times \mathbb{R}^{n\times k}$ with
\(
\|H\|_2^2+\|K\|_2^2=1
\)
and $t_0>0$ such that
\[
f(X^{\star}+tH,Y^{\star}+tK)=f(X^{\star},Y^{\star})-\gamma\,t^2
\quad\forall t\in [0,t_0],
\]
where
\[
\gamma:=\frac12\|P_XP_Y\|_{\mathrm{op}}\;>0
\]
and $P_X,P_Y\in \mathbb R^{k\times k}$ are orthogonal projectors onto $\ker(X^{\star})$ and $\ker(Y^{\star})$, respectively.
\end{theorem}
When \( k > 2r \), we always have \( \norm{P_X P_Y}_{\mathrm{op}} = 1 \), which implies that \( \gamma = \frac{1}{2} \). Moreover, unlike the case for smooth functions, the subgradient method can still converge to a saddle point of nonsmooth objectives such as \Eqref{eq:obj}. In our experiments, the subgradient method converges to a true solution even when \( k > r \), as shown in Figure \ref{fig:error-over}. However, the achieved accuracy is not as good as in the case \( k = r \) (see Figure \ref{fig:error-exact}), which aligns with the lack of sharpness suggested by Theorem \ref{thm:over}.

\begin{proof}
Since $\rank(X^{\star})=\rank(Y^{\star})=r$, there exist two orthonormal matrices $U_X,U_Y\in\mathbb{R}^{k\times k}$ such that
\[
X^{\star}U_X=\big[X_1\ \ 0\big],\qquad Y^{\star}U_Y=\big[Y_1\ \ 0\big],
\]
where $X_1\in\mathbb{R}^{m\times r}$ and $Y_1\in\mathbb{R}^{n\times r}$ have full column rank. Define the matrix $R:=U_Y^\top U_X$, and partition it as 
\[
R=\begin{bmatrix}R_{11}&R_{12}\\ R_{21}&R_{22}\end{bmatrix}.
\]
Then, the product $X^{\star}Y^{\star\top}$ becomes
\[
L = X^{\star}Y^{\star\top}=[X_1\ 0]R^\top[Y_1\ 0]^\top=X_1 R_{11}^\top Y_1^\top.
\]
This shows that $R_{11}$ is invertible.

Next, we prove that $R_{22} \neq 0$. Suppose instead that $R_{22}=0$. From the condition $RR^\top=I$, we get
\[
R_{21}R_{21}^\top + R_{22}R_{22}^\top = I_{k-r}\ \Rightarrow\ R_{21}R_{21}^\top = I_{k-r}.
\]
This implies that $P:=R_{21}^\top R_{21}$ is a nonzero orthogonal projection matrix. From $R^\top R=I$, we also have
\[
R_{11}^\top R_{11} + R_{21}^\top R_{21} = I_r \ \Rightarrow\ R_{11}^\top R_{11} = I_r - P.
\]
Since $P\neq 0$, the matrix $I_r-P$ is singular, contradicting the invertibility of $R_{11}$. Therefore, $R_{22}\neq 0$.

Let $w\in\mathbb{R}^{k-r}$ with $\|w\|_2=1$ be a right singular vector of $R_{22}$ corresponding to the largest singular value
$\|R_{22}\|_{\mathrm{op}}$, and define
\[
z:=\frac{R_{22}w}{\|R_{22}w\|_2^2}.
\]
Now fix an index pair $(i,j)$ such that $S_{ij}\neq 0$. Work in reduced coordinates
$\tilde H:=HU_Y=[H_1\ H_2]$, $\tilde K:=KU_X=[K_1\ K_2]$ and choose
\[
H_1=0,\quad K_1=0,\qquad
H_2=\alpha\,\sign(S_{ij}) e_{i} z^\top,\qquad
K_2=\beta\,e_{j} w^\top,
\]
with $\alpha,\beta>0$ to be determined later. We map back to the original coordinates using $H=\tilde H U_Y^\top$, $K=\tilde K U_X^\top$. A standard block computation gives
\[
H(Y^{\star})^\top+X^{\star}K^\top
=\tilde{H} [Y_1^\top~0] + [X_1~0]\tilde{K}^\top
= H_1Y_1^\top + X_1 K_1^\top = 0,
\]
and using $HK^\top=\tilde H R \tilde K^\top$ with $H_1=K_1=0$,
\[
HK^\top = H_2 R_{22} K_2^\top
= \alpha\beta\, \sign(S_{ij}) e_{i} (z^\top R_{22} w) e_{j}^\top
= \alpha\beta\, \sign(S_{ij}) e_{i}e_{j}^\top.
\]
Since $\|R_{22}w\|_2=\|R_{22}\|_{\mathrm{op}}$ and $\|z\|_2=1/\|R_{22}w\|_2$,
we normalize by choosing
\[
\alpha=\frac{1}{\sqrt{2}\,\|z\|_2},\qquad \beta=\frac{1}{\sqrt{2}},
\]
so that $\|H\|_2^2+\|K\|_2^2=1$ and
\[
\alpha\beta=\frac{1}{2\|z\|_2}=\frac12\|R_{22}\|_{\mathrm{op}}.
\]
Moreover, with $N_X:=U_X\!\begin{bmatrix}0\\ I_{k-r}\end{bmatrix}$ and
$N_Y:=U_Y\!\begin{bmatrix}0\\ I_{k-r}\end{bmatrix}$, we have orthonormal bases of
$\ker(X^\star)$ and $\ker(Y^\star)$ and
\[
R_{22}=N_Y^\top N_X,\qquad P_X=N_XN_X^\top,\qquad P_Y=N_YN_Y^\top,
\]
hence
\[
\|R_{22}\|_{\mathrm{op}}=\|N_Y^\top N_X\|_{\mathrm{op}}=\|P_XP_Y\|_{\mathrm{op}}.
\]
Therefore $\alpha\beta=\frac12\|P_XP_Y\|_{\mathrm{op}}=\gamma$.

Finally, for $t\in[0,t_0]$ with $t_0:=\sqrt{|S_{ij}|/\gamma}$,
\begin{align*}
f(X^{\star}+tH, Y^{\star}+tK)
&= \|t(H(Y^{\star})^\top+X^{\star}K^\top) + t^2 HK^\top - S\|_1\\
&= \|t^2 \gamma\, \sign(S_{ij}) e_{i}e_{j}^\top - S\|_1\\
&= \|S\|_1 - \gamma t^2
= f(X^{\star},Y^{\star}) - \gamma t^2,
\end{align*}
as claimed.
\end{proof}

\section{Conclusion and future directions}\label{sec:conclusion}

In this work, we studied the nonsmooth and nonconvex factorized robust PCA objective
\(
f(X,Y)=\|XY^\top-M\|_1
\)
under the standard low-rank plus sparse model \(M=L+S\), where \(L\) is incoherent and \(S\) has random sparse support.
Our main result certifies the \emph{optimality structure at the ground truth}: with high probability, every true rank-\(r\)
factorization \((X^\star,Y^\star)\in\mathcal L\) is a Clarke critical point.
We further characterized the local geometry. In the exactly-parameterized regime \(k=r\), the true solution set forms sharp
local minima, which implies local linear convergence of the subgradient method under standard step-size schedules.
In contrast, in the overparameterized regime \(k>r\), the same factorizations become strict saddle points.
A key technical ingredient is a restricted \(\ell_1\) operator-norm bound for \(\mathcal P_\Omega\) on the subspace
\(T\), obtained via a mixed-tail empirical process argument and generic chaining.
Below we highlight two directions for future work.

\begin{itemize}
    \item \emph{Extension to robust matrix completion.}
    Robust PCA corresponds to full observation of \(M\). In many applications, only a subset of entries is observed, leading to
    robust matrix completion.
    Prior work \cite[Theorem~6]{ma2025can} shows that true factorizations admit second-order descent directions in certain
    overparameterized regimes, but does not provide a first-order optimality certification.
    We believe the approach developed here—reducing first-order certification to controlling a restricted operator norm and then
    bounding an associated empirical process—can be adapted to this setting and help resolve this gap.
    Establishing such results would extend existing theoretical guarantees for nonsmooth factorized formulations beyond the fully observed setting.

    \item \emph{Are strict saddle points in the overparameterized regime active or inactive?}
    When \(k>r\), Theorem~\ref{thm:over} shows that true factorizations are strict saddle points: there exist directions along which the
    objective decreases quadratically. Yet in our experiments, subgradient methods with small random initialization often converge
    to the ground truth even when \(k>r\) (Figure~\ref{fig:error-over}).
    This suggests that these strict saddle points may be \emph{inactive} in the sense of
    \cite{davis2019active}, consistent with the conjecture in \cite{ma2025can}.
    Establishing this conjecture would help clarify the empirical behavior of first-order methods in the overparameterized regime.
\end{itemize}

\appendix
\section{Auxiliary lemmas}\label{sec:auxiliary}
\begin{lemma}[Bernstein's inequality, adapted from \cite{boucheron2003concentration}]
\label{lem::bernstein}
Let $X_1, \cdots, X_n$ be independent real-valued random variables satisfying $\bE [X_i]=0$ and $|X_i|\leq c$ for all $1\leq i\leq n$ and let $v=\sum_{i=1}^{n}\Var(X_i)$.
Then for all $u > 0$,
\begin{equation*}
    \mathbb{P}\left(\left|\sum_{i=1}^{n}X_i\right| \ge \sqrt{2vu}+\frac{c}{3}u\right) \le  2e^{-u}.
\end{equation*}
\end{lemma}

\begin{definition}[The $\gamma_\alpha$-functional]
A sequence $\mathcal{T} = (T_n)_{n \ge 0}$ of subsets of $T$ is called \emph{admissible} if 
\begin{equation*}
    |T_0| = 1 \quad \text{and} \quad |T_n| \le 2^{2^n} \ \text{for all } n \ge 1.
\end{equation*}
For any $0 < \alpha < \infty$, the \emph{$\gamma_\alpha$-functional} of $(T,d)$ is defined by
\begin{equation*}
    \gamma_\alpha(T, d)
= \inf_{\mathcal{T}} \sup_{t \in T} \sum_{n=0}^{\infty} 2^{n/\alpha} d(t, T_n),
\end{equation*}
where the infimum is taken over all admissible sequences $\mathcal{T}$.
\end{definition}

\begin{definition}[Mixed Subgaussian–Subexponential Increments]
\label{def::subgaussian-subexponential}
Let $d_1, d_2$ be two semi-metrics on a set $T$.  
A stochastic process $(X_t)_{t \in T}$ is said to have \emph{mixed subgaussian–subexponential increments}, or simply to \emph{have a mixed tail}, with respect to the pair $(d_1, d_2)$ if for all $s, t \in T$ and all $u \ge 0$,
\begin{equation*}
    \mathbb{P}\!\left( |X_t - X_s| \ge \sqrt{u}  d_2(t,s) + u  d_1(t,s) \right) \le 2 e^{-u}.
\end{equation*}
This means that the first part of the tail behaves as that of a subgaussian random variable, while the second part behaves as that of a subexponential random variable.
\end{definition}

\begin{theorem}[Adapted from \protect{\cite[Theorem~3.5]{dirksen2015tail}}]
\label{thm:mixed-tail}
If $(X_t)_{t \in T}$ has a mixed tail with respect to $(d_1, d_2)$, then there exists a universal constant $C > 0$ such that for any fixed $t_0 \in T$,
\begin{equation*}
    \bE \left[\sup_{t \in T} |X_t - X_{t_0}|\right]
\le C \Bigl( \gamma_2(T, d_2) + \gamma_1(T, d_1) \Bigr)
+ 2  \sup_{t \in T} \bE \left[|X_t - X_{t_0}|\right].
\end{equation*}
As a consequence, there are constants $c,C >0$ such that for any $u\geq 1$,
\begin{equation*}
\mathbb{P}\Big(\sup _{t \in T}\left\|X_t-X_{t_0}\right\| \geq C\left(\gamma_2\left(T, d_2\right)+\gamma_1\left(T, d_1\right)\right)+c\left(\sqrt{u} \Delta_{d_2}(T)+u \Delta_{d_1}(T)\right)\!\Big) \leq e^{-u}.
\end{equation*}
Here $\Delta_d(T)=\sup_{s, t\in T}d(s, t)$ is the diameter of $T$ under metric $d$.
\end{theorem}

\begin{lemma}[Adapted from \protect{\cite[Equation~2.3]{dirksen2015tail}}]
\label{lem::dudley}
    For any $\varepsilon>0$ let $N(T,d,\varepsilon)$ denote the covering number of $T$, i.e., the smallest
number of balls of radius $\varepsilon$ in $(T,d)$ needed to cover $T$. We have
\begin{equation*}
    \gamma_{\alpha}(T,d)\leq C_\alpha \int_0^{\infty}\left(\log\left(N(T,d,\varepsilon)\right)\right)^{1/\alpha}\operatorname{d}\!\varepsilon,
\end{equation*}
where $C_\alpha>0$ is a constant that only depends on $\alpha$.
\end{lemma}

\begin{lemma}[Maximal inequality]
\label{lem::maximal-ineq}
    Suppose $X_1, \cdots, X_n$ are i.i.d. sub-Gaussian random variables with variance proxy $\sigma^2$, then
    \begin{equation*}
        \bE\left[\max_{1\leq i\leq n}X_i\right]\leq \sqrt{2\sigma^2 \log(n)}.
    \end{equation*}
\end{lemma}

\begin{lemma}[\protect{Talagrand’s majorizing measure theorem \cite[Theorem 8.5.5]{vershynin2018high}}]
\label{lem::majorizing-measure}
    Let $(X_t)_{t\in T}$ be a mean-zero
Gaussian process on a set $T$. Consider the canonical metric defined on $T$ defined by
$d(t, s) = \sqrt{\bE[(X_t-X_s)^2]}$. Then, there exists a universal constant $C>0$ such that
\begin{equation*}
     \gamma_2(T , d)\leq C\bE\left[\sup_{t\in T}
X_t\right].
\end{equation*}
\end{lemma}

\bibliographystyle{alpha}
\bibliography{mybib}
\end{document}